\documentclass[10pt]{amsart}
\usepackage[pagebackref]{hyperref}

\renewcommand{\d}{\mathrm{d}}%exterior derivative

\newcommand{\bbC}{{\mathbb C}}

\newcommand{\bbP}{{\mathbb P}}

\newcommand{\SO}{\operatorname{SO}}

\DeclareMathOperator{\Hom}{Hom}

\newcommand{\ssf}{\mathsf{s}}
\newcommand{\bsf}{\mathsf{b}}
\newcommand{\tsf}{\mathsf{t}}
\newcommand{\psf}{\mathsf{p}}
\newcommand{\zsf}{\mathsf{z}}

\newcommand{\la}{\langle}
\newcommand{\ra}{\rangle}

\newcommand{\Gtwo}{\mathrm{G}_2}

\newcommand{\w}{{\mathchoice{\,{\scriptstyle\wedge}\,}{{\scriptstyle\wedge}}
      {{\scriptscriptstyle\wedge}}{{\scriptscriptstyle\wedge}}}}
\newcommand{\lhk}{\mathbin{\hbox{\vrule height1.4pt width4pt depth-1pt 
             \vrule height4pt width0.4pt depth-1pt}}}

\newcommand{\be}{\begin{equation}}
\newcommand{\ee}{\end{equation}}
\newcommand{\bpm}{\begin{pmatrix}}
\newcommand{\epm}{\end{pmatrix}}

\numberwithin{equation}{section}

\newtheorem{theorem}{Theorem}

\newtheorem{proposition}{Proposition}
\newtheorem{corollary}{Corollary}

\theoremstyle{remark}
\newtheorem{definition}{Definition}

\newtheorem{remark}{Remark}

\begin{document}

\author[R. Bryant]{Robert L. Bryant}
\address{Duke University Mathematics Department\\
         PO Box 90320\\
         Durham, NC 27708-0320}
\email{\href{mailto:bryant@math.duke.edu}{bryant@math.duke.edu}}
\urladdr{\href{http://www.math.duke.edu/~bryant}%
         {http://www.math.duke.edu/\lower3pt\hbox{\symbol{'176}}bryant}}

\title[Closed $\Gtwo$ Solitons]
      {The generality of closed $\Gtwo$ solitons}

\date{April 30, 2022}

\begin{abstract}
The local generality of the space of solitons 
for the Laplacian flow of closed $\Gtwo$-structures 
is analyzed, and it is shown that the germs of such
structures depend, up to diffeomorphism, 
on 16 functions of 6 variables (in the sense of \'E.~Cartan).
The method is to construct a natural exterior differential system
whose integral manifolds describe such solitons and to show
that it is involutive in Cartan's sense, so that Cartan-K\"ahler
theory can be applied. 

Meanwhile, it turns out that, for the more special case of
\emph{gradient} solitons, the natural exterior differential system
is \emph{not} involutive, and the generality of these structures
remains a mystery.
\end{abstract}

\subjclass{
 53E99,  %Geometric evolution equations
 53C29%  Issues of holonomy in differential geometry
}

\keywords{$G_2$-structures, solitons}

\thanks{
Thanks to the Simons Foundation for its support 
via the Simons Collaboration Grant 
``Special Holonomy in Geometry, Analysis, and Physics''. 
}

\maketitle

\emph{Dedicated, with much admiration and gratitude for his many years of  advice, encouragement, and friendship, to H.~Blaine Lawson, Jr.~on the occasion of his 
80th birthday.}

\setcounter{tocdepth}{2}
\tableofcontents

\section{Introduction}\label{sec: intro}
For the necessary background on $\Gtwo$-structures and the notation and conventions that this article uses, the reader may consult \cite{Bryant0}
and, especially,~\cite{Bryant1}.

\subsection{$\Gtwo$-structures} 
A (smooth) $\Gtwo$-structure on a $7$-manifold~$M$ is a (smooth) 
$3$-form $\phi\in\Omega^3(M)$ that is \emph{definite} in the sense that, 
for any nonzero tangent vector $v\in T_xM$, 
the $7$-form~$(v\lhk \phi)\w(v\lhk \phi)\w\phi\in \Lambda^7(T^*_xM)$ 
is also nonzero.  The space of definite $3$-forms on $M^7$ 
will be denoted $\Omega^3_+(M)$.  

Given a $\sigma\in\Omega^3_+(M)$, there is a unique metric $g_\sigma$ and orientation $\ast_\sigma$ (i.e., the Hodge star operator defined by $g_\sigma$ 
and the orientation) such that
\be
(X\lhk \sigma)\w(Y\lhk \sigma)\w\sigma = 6\la X,Y\ra_{g_\sigma}\,{\ast_\sigma} 1
\ee
for all vector fields $X$ and $Y$ on $M$.

It is not difficult to prove the first-order $6$-form identity~\cite{Bryant0}
\be\label{FundIdent}
\bigl({\ast_\sigma}{\mathrm{d}}({\ast_\sigma}\sigma)\bigr)\w{\ast_\sigma}\sigma
+ ({\ast_\sigma}\mathrm{d}\sigma) \w\sigma = 0.
\ee
In particular, the equations $\d\sigma = \d({\ast_\sigma}\sigma) = 0$ represent only $35 + 21 - 7 = 49$ first-order equations on a $\Gtwo$-structure~$\sigma$.  (It is easy to see that these $49$ equations are independent, i.e., there are no further algebraic relations between $\sigma$, $\d\sigma$ and ${\ast_\sigma}{\mathrm{d}}({\ast_\sigma}\sigma)$ beyond those implied by \eqref{FundIdent}.)

\subsection{Relation with holonomy}
Fern\'andez and Gray~\cite{FerGray} proved that, if $\sigma$ and ${\ast_\sigma}\sigma$ are both closed forms, then $\sigma$ is $g_\sigma$-parallel. This, in particular, implies that the holonomy of $g_\sigma$ is isomorphic to a subgroup of $\Gtwo\subset\SO(7)$.  Conversely, if the holonomy of a metric $g$ on $M^7$ 
is contained in $\Gtwo$, then there exists a definite $3$-form $\sigma$ 
(unique up to replacement by $-\sigma$ if the holonomy of $g$ 
is isomorphic to $\Gtwo$) satisfying $\d\sigma = \d({\ast_\sigma}\sigma) = 0$ and such that $g = g_\sigma$.  This characterization has been essential in all constructions of metrics on $7$-manifolds with holonomy $\Gtwo$, indeed, in even proving their existence.

While it is easy to write down a local formula for the `generic' closed $\Gtwo$-structure~$\sigma$, the full equations $\d\sigma = \d{\ast_\sigma}\sigma = 0$ are highly nonlinear (as well as being overdetermined).  In~\cite{Bryant0}, it was shown how to interpret these equations as an involutive exterior differential system (see~\cite{Bryant2}) and analyze their local generality in the sense of \'Elie Cartan.  In particular, it was shown that the `general' solution, up to diffeomorphism, depends on six functions of six variables.  

One way of expressing this in terms that are, perhaps, more familiar is the following:  Suppose that one considers the space $\mathcal{J}^k$ of $k$-jets of germs of solutions of $\d\sigma = \d{\ast_\sigma}\sigma = 0$ on a neighborhood of $0\in\mathbb{R}^7$ with the property that the standard coordinates on $\mathbb{R}^7$ are geodesic normal coordinates for $g_\sigma$ at~$0$ and the value of $\sigma$ at $0$ takes a standard form\footnote{This nearly removes the diffeomorphism ambiguity. In order to completely remove it, one would have to quotient by the natural action of $\Gtwo$ on these jet spaces.}.  Let $\delta_k = \dim \mathcal{J}^k - \dim \mathcal{J}^{k-1}$ and consider the Poincar\'e series $P(t) 
= \delta_0 + \delta_1\,t + \delta_2\,t^2 + \cdots$.  Then 
\be
P(t) = t^2\left(\frac{14}{(1{-}t)^2} +\frac{21}{(1{-}t)^3}
+\frac{21}{(1{-}t)^4}+ \frac{15}{(1{-}t)^5} +\frac{6}{(1{-}t)^6} \right).
\ee
In particular, the `leading term' contributing to the growth of $\delta_k$, 
namely $6t^2/(1{-}t)^6$, is the same, up to a degree shift, as the `leading term' 
describing the growth of $k$-jets of $6$ functions of $6$ variables. 

\subsection{Solitons}
D.~Joyce pioneered the approach to constructing compact examples of metrics with holonomy $\Gtwo$ by starting with a `soft' construction of a closed $\Gtwo$-structure $\sigma$ on certain $7$-manifolds with the property that $\d({\ast_\sigma}\sigma)$ had sufficiently small norm (in comparison with other geometric features of $g_\sigma$) and showing that, in such situations, 
one could `perturb' the given closed $\Gtwo$-structure $\sigma$ 
to a nearby one that was also co-closed.  (See the fundamental reference~\cite{Joyce1} and the original papers cited therein for details.)

The many successes of R.~Hamilton's program for using the Ricci flow in Riemannian geometry inspired the idea of canonically `improving' a closed $\Gtwo$-structure $\sigma_0$ on a compact $7$-manifold~$M$ by considering the so-called `Laplacian flow'
\be
\frac{d\,\sigma}{dt} = \Delta_{g_\sigma}\,\sigma
\ee
with $\sigma_0$ as initial condition.  In particular, a fixed point of this flow, i.e., a closed $\Gtwo$-structure on $M$ that satisfies $\Delta_{g_\sigma}\,\sigma = 0$, would satisfy $\d\sigma = \d{\ast_\sigma}\sigma=0$ and hence lead to a metric $g_\sigma$ with holonomy contained in $\Gtwo$.  It has been shown~\cite{BryantXu} that the closed $\Gtwo$-Laplacian flow on a compact $7$-manifold has short-time existence and uniqueness.

Meanwhile, Hitchin~\cite{Hitchin1,Hitchin2} observed that the above Laplacian flow is (up to a constant factor) the gradient flow of the \emph{volume functional}
$$
\mathcal{V}(\sigma) = \int_{M} \sigma\w{\ast_\sigma}\sigma
$$
on the space $\mathcal{Z}(M,\gamma)\subset\Omega^3_+(M)$ of closed $\Gtwo$-structures in a fixed deRham cohomology class~$\gamma\in H^3_{dR}(M)$.  Hitchin has shown that any critical points of $\mathcal{V}$ in $\mathcal{Z}(M,\gamma)\subset\Omega^3_+(M)$ must be local maxima.  

However, there need not be any critical points even if the flow exists for all time, as the example in Remark~18 of~\cite{Bryant1} shows.

It is not clear just what sort of singularities one might encounter in the Laplacian flow, and one strategy for gaining some intuition, based on the strategy that worked in the case of Ricci flow, is to study the so-called `soliton' solutions.

\begin{definition}
Let $\lambda$ be a fixed real constant.
A pair $(\sigma,X)$, where $\sigma$ is a closed $\Gtwo$-structure on $M^7$ and $X$ is a vector field on $M^7$, is a \emph{$\lambda$-soliton} for the $\Gtwo$-Laplacian flow if 
\be
\Delta_\sigma\,\sigma = \lambda\,\sigma + \mathcal{L}_X\sigma\,.
\ee
If, in addition, $X= \nabla^{g_\sigma}f$ for some function $f$ on $M$, the pair $(\sigma, X)$ will be said to be a \emph{gradient $\lambda$-soliton}.
\end{definition}

It is not difficult to show that, if $(\sigma, X)$ is a $\lambda$-soliton for the $\Gtwo$-Laplacian flow on $M$ and $\Phi_\tau:M\to M$ is the time-$\tau$ flow of $X$, then the time-dependent $\Gtwo$-structure
\be
\phi(t) = \mathrm{e}^{\lambda t} \Phi^*_{f(t,\lambda)}\sigma
\ee
where $f(t,0) = t$ and $f(t,\lambda) = (1-\mathrm{e}^{-\lambda t})/\lambda$ when $\lambda\not=0$, satisfies the initial condition $\phi(0) = \sigma$
and the Laplacian flow equation
\be
\frac{d\,\phi}{dt} = \Delta_{g_\phi}\,\phi\,.
\ee

Many explicit examples of such $\lambda$-solitons are now known.  Most of the known examples are constructed using some version of `dimensional reduction', i.e., assuming some sort of symmetry, thereby reducing the problem to geometric data on a lower dimensional manifold, see \cite{Ball2020,Fowdar,HaskinsNordstrom2021,Lauret2017,Lotay2017}. 

The goal of this article is to investigate the local generality, in \'E.~Cartan's sense, of the space of $\lambda$-solitons.  A straightforward count of equations reveals that, locally, this is 70 equations for the 42 unknowns in $(\sigma, X)$, so the local generality of the solutions of these equations is not at all clear.  Moreover, the system is degenerate (in particular, every direction is characteristic) because it is evidently invariant under diffeomorphisms.

\section{Basic Identities}\label{BI}

In this section, I will assume that $(\sigma, X)$ is a $\lambda$-soliton for the $\Gtwo$-Laplacian flow equation on $M^7$ and explain how to express this condition  naturally in terms of an exterior differential system.

\subsection{Expression in differential forms}  
Because of the fundamental identity~\eqref{FundIdent}, the closedness of $\sigma$ implies that $\bigl({\ast_\sigma}\d({\ast_\sigma}\sigma)\bigr) \w {\ast_\sigma}\sigma = 0$, so there is a $2$-form $\tau$ that satisfies
\be\label{tauDefined}
\d({\ast_\sigma}\sigma) = \tau\w \sigma = - {\ast_\sigma}\tau.
\ee
(See \S4.6 of~\cite{Bryant1}.)  Moreover, $\tau$ satisfies
\be
\d\tau = \Delta_\sigma\sigma = \lambda\,\sigma + \mathcal{L}_X\sigma
=  \lambda\,\sigma + \d(X\lhk\sigma),
\ee
which can be rewritten as
\be\label{SolitonRephrased}
\d(\tau - X\lhk\sigma) = \lambda\,\sigma.
\ee

As explained in~\S2.6 of \cite{Bryant1}, every $2$-form $\beta\in\Omega^2(M)$ can be written uniquely as a sum $\beta = \beta_7 + \beta_{14}$ where ${\ast_\sigma}\beta_{14} = -\beta_{14}\w\sigma$ while ${\ast_\sigma}\beta_7 = \tfrac12\,\beta_7\w\sigma$.  In what follows, set $\beta = \tau - X\lhk\sigma$, so that $\beta_{14} = \tau$ and $\beta_7 = -X\lhk\sigma$.  Then \eqref{SolitonRephrased} can be rewritten as
the $3$-form equation
\be\label{SolitonRewrite}
0 = \d\beta - \lambda\,\sigma.
\ee
We already have the 4-form equation
\be\label{SigmaClosed}
0 = \d\sigma
\ee
since $\sigma$ is closed.  Using the above algebraic identities, 
\eqref{tauDefined} can be written as a 5-form equation
\be\label{tauRedefined}
0 = \d({\ast_\sigma}\sigma) - \tau\w\sigma = \d({\ast_\sigma}\sigma)
+ \tfrac23\,{\ast_\sigma}\beta - \tfrac13\,\beta\w\sigma.
\ee
Finally, taking the exterior derivative of this last equation and using the equations found so far, we have the 6-form equation
\be\label{d of tauRedefined}
0 = \d(\ast_\sigma\beta).
\ee

The virtue of the equations (2.4--7) is that they only involve the constant~$\lambda$, the definite $3$-form~$\sigma$ (and its algebraically associated Hodge star operator), and the $2$-form~$\beta$.  In fact, this is enough to recover the structure of a $\lambda$-soliton for the closed $\Gtwo$-Laplacian flow.

\begin{proposition}
Let $\lambda$ be a constant, and let $\sigma\in\Omega^3_+(M)$ 
and $\beta\in\Omega^2(M)$ satisfy equations (2.4--7).  
Then for the unique vector field $X$ on $M$ 
such that $\beta_7 = - X\lhk \sigma$, 
the pair $(\sigma, X)$ 
is a $\lambda$-soliton for the closed $\Gtwo$-Laplacian flow.  
Conversely, if $(\sigma,X)$ 
is a $\lambda$-soliton for the closed $\Gtwo$-Laplacian flow, 
then setting $\beta = -\ast_\sigma\d({\ast_\sigma}\sigma)- X\lhk\sigma$
defines a pair~$(\sigma,\beta)$ that satisfies (2.4--7).
\end{proposition}

\begin{proof}
This is a matter of unwinding the definitions.  We have $\d\sigma=0$,
and then equation (2.6) implies that $\tau=\beta_{14}$ and that $\d({\ast_\sigma}\sigma) = \tau\w\sigma$, so that $\d\tau = \Delta_\sigma \sigma$.  Then (2.4) becomes (2.3).  The converse follows in the same way.
\end{proof}

\begin{remark}
The reader may wonder why \eqref{d of tauRedefined} was included, 
since it is a consequence of the other three equations just by applying 
the exterior derivative to \eqref{tauRedefined}.  Moreover, when $\lambda$ is nonzero, \eqref{SigmaClosed} follows from \eqref{SolitonRewrite}, so it might seem that \eqref{SigmaClosed} could be omitted as well, at least when $\lambda$ is nonzero.

The reason is that I want to make explicit \emph{all} 
of the first-order equations satisfied by the pair $(\sigma,\beta)$.
The four equations (2.4--7) are the vanishing of a $3$-form, a $4$-form, a $5$-form, and a $6$-form, which is potentially
$$
35 + 35 + 21 + 7 = 98 
$$
independent first order equations.  However, the identity \eqref{FundIdent},
shows that these 98 equations are \emph{not} independent; they `overlap' by at least $7$ equations.  

In fact, the overlap is \emph{exactly} the $7$ equations of the fundamental identity~\eqref{FundIdent} since it is easy to see that, for a generic $\sigma\in\Omega^3_+(M)$ and $\beta\in\Omega^2(M)$, the only relation among the forms
$$
\d\beta,\  \d\sigma,\ \d(\ast_\sigma\sigma),\  \d(\ast_\sigma\beta)
$$
is the fundamental identity relating the middle two.  Thus, the equations (2.4--7) constitute exactly $91$ independent (quasi-linear) first order equations on the pair $(\sigma,\beta)$, a fact to which I will return.

Indeed, I am going to argue that this system of $91$ equations 
is \emph{involutive} in \'E.~Cartan's sense, 
so that the Cartan-K\"ahler Theorem can be applied 
to prove local existence of solutions and describe their `degree of generality'.  By contrast, the `determined' system of $35+21 = 56$ first-order equations 
for $(\sigma,\beta)$ represented by combining \eqref{SolitonRewrite} 
and \eqref{tauRedefined} is certainly \emph{not} involutive.
\end{remark}

\section{Formulation as an exterior differential system}

I now want to explain how the above system of 91 first-order partial differential equations for the pair $(\sigma,\beta)$, which is a section of a smooth bundle of fiber rank $35+21 = 56$ over~$M^7$, can be interpreted as defining an exterior differential system for such pairs.

\subsection{A differential ideal}
Let $\Lambda^3_+(T^*M)\subset\Lambda^3(T^*M)$ 
denote the open subset of the bundle of $3$-forms on $M^7$ 
whose (local) sections are the (local) $\Gtwo$-structures on~$M$.  
Similarly, let $\Lambda^2(T^*M)$ denote the bundle of $2$-forms on $M$.

Let $X\subset \Lambda^3(T^*M)\oplus\Lambda^2(T^*M)$ 
denote the set of pairs $(s,b)$ with $s\in \Lambda^3_+(T^*_mM)$ 
and $b\in \Lambda^2(T^*_mM)$ for some $m\in M$.  
This $X$ is open in $\Lambda^3(T^*M)\oplus\Lambda^2(T^*M)$ 
and is a smooth fiber bundle over~$M$ 
with a submersive base projection $\pi:X\to M$, with $\pi(s,b) = m$.
Let $\pi^*:\Lambda^k(T^*M)\to \Lambda^k(T^*X)$ 
denote the induced `pullback' mapping.

There exist natural `tautological' forms on $X$ that are defined as follows:  
For each $(s,b)\in X$, let $\ssf_{(s,b)} = \pi^*(s)$ 
and let $\bsf_{(s,b)} = \pi^*(b)$.  
Then $\ssf$ (respectively, $\bsf$) is a smooth $3$-form 
(respectively, $2$-form) on~$X$ 
with the `reproducing' property that, 
for any section $(\sigma,\beta):M\to X$, there holds
$(\sigma,\beta)^*(\ssf) = \sigma$ and $(\sigma,\beta)^*(\bsf) = \beta$.
Moreover, because exterior derivative commutes with pullback, 
the identities
$(\sigma,\beta)^*(\d\ssf) = \d\sigma$ 
and $(\sigma,\beta)^*(\d\bsf) = \d\beta$ hold as well.

I will need an extension of this construction:  
Since, for any $s\in \Lambda^3_+(T^*_mM)$, 
there is a well-defined Hodge star operator 
$\ast_s:\Lambda^p(T^*_mM)\to \Lambda^{7-p}(T^*_mM)$, 
there also exists a smooth $4$-form $\tsf$ 
(respectively, $5$-form, $\psf$) on $X$ with the property that
$\tsf_{(s,b)} = \pi^*(\ast_s s)$ 
(respectively, $\psf_{(s,b)} = \pi^*(\ast_s b)$).  
These two forms have the `reproducing' property  that,
for any section $(\sigma,\beta):M\to X$, there holds
$(\sigma,\beta)^*(\tsf) = {\ast_\sigma}\sigma$ 
and $(\sigma,\beta)^*(\psf) = {\ast_\sigma}\beta$.

Finally, there exists a smooth $7$-form $\omega$ on $X$ 
with the property that $\omega_{(s,b)} = \pi^*(\ast_s 1)$ 
that will be used below, so I introduce it now.  

Note that all of the forms $\ssf$, $\tsf$, $\bsf$, $\psf$, and $\omega$ 
are $\pi$-semibasic, i.e., their interior product 
with any $\pi$-vertical vector field on $X$ vanishes identically.

Now, by construction, for any constant $\lambda$ and any section $(\sigma,\beta):M\to X$, one has
\be\label{PullbackFormulae}
\begin{aligned}
(\sigma,\beta)^*(\d\bsf - \lambda\,\ssf) &= \d\beta - \lambda\,\sigma\\
(\sigma,\beta)^*(\d\ssf) &= \d\sigma \\
(\sigma,\beta)^*(\d\tsf + \tfrac23\,\psf - \tfrac13\,\bsf\w\ssf) 
&= \d({\ast_\sigma}\sigma) +\tfrac23\,{\ast_\sigma}\beta -\tfrac13\,\beta\w\sigma\\
(\sigma,\beta)^*(\d\psf) &= \d({\ast_\sigma}\beta)
\end{aligned}
\ee

This motivates defining the differential forms $\Upsilon_i\in\Omega^i(X)$,
\be\label{Upsilondefinitions}
\begin{aligned}
\Upsilon_3 &= \d\bsf - \lambda\,\ssf,\\
\Upsilon_4 &= \d\ssf, \\
\Upsilon_5 &= \d\tsf + \tfrac23\,\psf - \tfrac13\,\bsf\w\ssf,\\
\Upsilon_6 &= \d\psf.
\end{aligned}
\ee
Note the identities
\be
\begin{aligned}
\d\Upsilon_3 &= -\lambda\,\Upsilon_4\,,\\
\d\Upsilon_4 &= 0,\\
\d\Upsilon_5 &= \tfrac23\,\Upsilon_6 -\tfrac13\,\Upsilon_3\w\ssf
-\tfrac13\,\bsf\w \Upsilon_4\,,\\
\d\Upsilon_6 &= 0,
\end{aligned}
\ee
which imply that the algebraic ideal $\mathcal{I}_\lambda$ in $\Omega^*(X)$ generated by $\Upsilon_3, \Upsilon_4, \Upsilon_5, \Upsilon_6$ is differentially closed.  Consequently, $\mathcal{I}_\lambda$ is an exterior differential system on $X$. 

The interest in $\mathcal{I}_\lambda$ is explained by the following result.

\begin{proposition}
A section $(\sigma,\beta):M\to X$ is an integral manifold of $\mathcal{I}_\lambda$ if and only if it comes from a $\lambda$-soliton for the closed $\Gtwo$-Laplacian flow.  In particular, a $7$-dimensional integral manifold of $\mathcal{I}_\lambda$ that is transverse to the fibers of $\pi:X\to M$ is locally the graph of a section $(\sigma,\beta):M\to X$ that comes from a $\lambda$-soliton for the closed $\Gtwo$-Laplacian flow.
\end{proposition} 

\begin{proof}
It has been shown that a pair $(\sigma,\beta)$ 
comes from a $\lambda$-soliton for the closed $\Gtwo$-Laplacian flow 
if and only if the pair satisfies the equations (2.4--7).  
Since this vanishing is exactly the vanishing of the righthand sides of~\eqref{PullbackFormulae}, 
the definitions given in \eqref{Upsilondefinitions} 
and the definition of $\mathcal{I}_\lambda$ 
show that this vanishing 
is equivalent to the section $(\sigma,\beta)$ 
being a $7$-dimensional integral manifold 
(necessarily transverse to the $\pi$-fibers) of $\mathcal{I}_\lambda$.  
The rest of the statement follows immediately.
\end{proof}

\subsection{Involutivity}
I can now state the main result of this article.

\begin{theorem}\label{involutivity}
The exterior differential system $\mathcal{I}_\lambda$ 
with independence condition $\omega$ is involutive, 
with Cartan characters
$$
(s_0,s_1,s_2,s_3,s_4,s_5,s_6,s_7) = (0,0,1,3,7,15,23,7).
$$
\end{theorem}

\begin{proof}
The proof will be by Cartan's Test for involutivity 
(see, for example,~\cite{Bryant2}).  

The first thing to check for each $(s,b)\in X$ 
is the dimension of the space of $7$-dimensional integral elements 
$E\subset T_{(s,b)}X$ of $\mathcal{I}_\lambda$ 
on which $\omega$ is non-vanishing.  
I claim that this dimension is $301 = 392-91$.  

To see this, choose a splitting $T_{(s,b)}X = \ker\pi'(s,b) + W$,
where $\pi'(s,b):W\to T_{\pi(s,b)}M$ is an isomorphism.  Then an integral element $E\subset T_{(s,b)}X$ on which $\omega$ is non-vanishing can be thought of as the graph of a linear map $L:W\to \ker\pi'(s,b)$, i.e., it can be identified with an element of the vector space $\Hom(W,\ker\pi'(s,b))$.  Because the four tautological forms $\ssf$, $\bsf$, $\tsf$, and $\psf$ are $\pi$-semi-basic, it follows that the induced mapping of $\Hom(W,\ker\pi'(s,b))$ into the 98-dimensional vector space
$$
\Lambda^3(W^*)\oplus\Lambda^4(W^*)\oplus\Lambda^5(W^*)\oplus\Lambda^6(W^*)
$$
induced by evaluating $(\Upsilon_3,\Upsilon_4,\Upsilon_5,\Upsilon_6)$
on the graph of $L\in \Hom(W,\ker\pi'(s,b))$ is an affine mapping. Moreover, its image lies in the 91-dimensional subspace cut out by the fundamental identity.  Finally, looking at the `leading terms' of $(\Upsilon_3,\Upsilon_4,\Upsilon_5,\Upsilon_6)$, i.e., $(\d\bsf,\d\ssf,\d\tsf,\d\psf)$ and considering their interpretation as 
$$
\bigl(\d\beta,\d\sigma,\d(\ast_\sigma\sigma),\d(\ast_\sigma\beta)\bigr),
$$ 
one sees that the leading order part is a linear mapping that is surjective onto this 91-dimensional subspace.  

Consequently, the set~$\mathcal{V}(\mathcal{I}_\lambda,\omega)$ 
of admissible integral elements of $\mathcal{I}_\lambda$ at $(s,b)\in X$ 
is a (nonempty) affine subspace of $\Hom\bigl(W,\ker\pi'(s,b)\bigr)$ 
of dimension $S = 56\times7 - 91 = 392 - 91 = 301$.
Indeed $\mathcal{V}(\mathcal{I}_\lambda,\omega)$ is a smooth submanifold
of $\mathrm{Gr}(7,TX)$, the Grassmann bundle of $7$-dimensional subspaces
of the tangent spaces to $X$.

It remains to compute the Cartan characters and apply Cartan's Test.  
Now that it has been established that the space 
of admissible integral elements of $\mathcal{I}_\lambda$ 
at each point is nonempty, the Cartan characters 
can be computed by considering just the leading terms listed above.

First, recall that it was established in~\cite{Bryant0} 
that the Cartan characters of the (involutive) equation 
$D(\sigma) = \bigl(\d\sigma,\d(\ast_\sigma\sigma)\bigr) = (0,0)$ 
relative to any flag are
$$
(s_0,s_1,s_2,s_3,s_4,s_5,s_6,s_7) = (0,0,0,1,4,10,13,7),
$$
Meanwhile, once $\sigma$ is fixed, the characters 
of the involutive first-order system 
$D(\beta)=\bigl(\d\beta,\d(\ast_\sigma\beta)\bigr) = (0,0)$ for $\beta$ 
a $2$-form on a manifold with a specified $\Gtwo$-structure $\sigma$ are
$$
(s_0,s_1,s_2,s_3,s_4,s_5,s_6,s_7) = (0,0,1,2,3,5,10,0).
$$
It follows that there exists an integral flag of the combined system
with characters
$$
(s_0,s_1,s_2,s_3,s_4,s_5,s_6,s_7) = (0,0,1,3,7,15,23,7).
$$
It follows that $\sum_{k=0}^7 k\,s_k = 301 =S$, the dimension of the space of integral elements of $\mathcal{I}_\lambda$ at $(s,b)$.  
Thus, equality holds in Cartan's Test, which implies that such a flag 
is Cartan-regular. Consequently, the system is involutive, as was to be shown.
\end{proof}

\begin{remark}\label{analyticity}
Since $M$ is assumed to be smooth, it also carries a real-analytic structure (in fact, many, but all of them are equivalent), and $X$ will naturally inherit a real-analytic structure once one is chosen on $M$.  Because of the nature of their construction, the forms $\Upsilon_3, \Upsilon_4, \Upsilon_5, \Upsilon_6$ will be real-analytic with respect to the real-analytic structure on $X$.  Hence $\mathcal{I}_\lambda$ is real-analytic with respect to any real-analytic structure on $X$.
\end{remark}

\begin{corollary}\label{cor: existence}
For any $(s,b)\in X$ with $\pi(s,b) = m\in M$, there exists a $\lambda$-soliton $(\sigma,\beta)$ for the closed $\Gtwo$-Laplacian flow on an $m$-neighborhood $U\subset M$ such that $(\sigma_m,\beta_m) = (s,b)$.  In fact, for any $7$-dimensional integral element $E\subset T_{(s,b)}X$ of $\mathcal{I}_\lambda$ on which $\omega$ is non-vanishing, there exists such a $(\sigma,\beta):U\to X$ such that $(\sigma,\beta)'(s,b)(T_mM) = E$.
\end{corollary}

\begin{proof}
In light of Theorem~\ref{involutivity} and Remark~\ref{analyticity}, 
Corollary~\ref{cor: existence} follows immediately 
from the Cartan-K\"ahler Theorem.
\end{proof}

\begin{remark}[Generality in Harmonic Coordinates]
Since the last nonzero character of $\mathcal{I}_\lambda$ is $s_7=7$, the na\"\i ve statement is that local $\lambda$-solitons for the closed $\Gtwo$-Laplacian flow depend on `seven functions of seven variables' in Cartan's terminology\footnote{The Cartan characters $s_i$ indicate how much essential choice there is in building up a solution to the equations via a sequence of (possibly underdetermined) Cauchy problems.}.  However, this reflects the fact that the equations are invariant under the full diffeomorphism group in dimension~$7$, which itself is described (locally) by specifying seven functions of seven variables.

There are a number of methods of, at least partially, `normalizing' the local solutions by choosing special coordinates.  The simplest one perhaps, is to ask how general the local solutions are when presented in harmonic coordinates for the associated metric $g_\sigma$.  One way to see this is to specialize to the case $M=\mathbb{R}^7$ with standard coordinates $u=(u^i)$ and impose the conditions $\d\bigl(\ast_\sigma(\d u^i)\bigr) = 0$.  This can be added to the exterior differential system by defining the seven $6$-forms 
\be
\Phi^i = \pi^*\bigl(\ast_s(\d u^i)\bigr)\in\Lambda^6(T^*_{(s,b)}X)
\ee
and then defining the extended ideal $\mathcal{I}_\lambda^+$ generated by $\mathcal{I}_\lambda$ and the seven closed $7$-forms $\d\Phi^i$.  Following the same argument as that used in Theorem~\ref{involutivity}, one can show that the system $\mathcal{I}_\lambda^+$ on $X$ is also involutive, but now with characters
\be
(s_0,s_1,s_2,s_3,s_4,s_5,s_6,s_7) = (0,0,1,3,7,15,30,0).
\ee

One can interpret the local integrals of $\mathcal{I}_\lambda^+$ as $\lambda$-solitons for the closed $\Gtwo$-Laplacian flow that are presented in harmonic coordinates.  In Cartan's terminology, such solitons depend locally on $s_6 = 30$ functions of six variables.  However, the choice of a local harmonic function, given $\sigma$, depends on 2 functions of six variables, so a choice of a harmonic coordinate system depends on $2\times 7 = 14$ functions of six variables.  Thus, one can argue that local solitons $(\sigma,\beta)$ depend on $16 = 30-14$ functions of six variables, in the sense that the rate of growth of the dimensions of the $k$-jets of solutions up to diffeomorphism is the same as the rate of growth of 16 functions of 6 variables.
\end{remark}

\begin{remark}[Analyticity]
Another reason for considering the presentation of $\lambda$-solitons $(\sigma,\beta)$ in harmonic coordinates is that it makes clear the regularity of such structures.  Because the last nonzero character of $\mathcal{I}_\lambda^+$ is $s_6 = 30$, it follows~\cite[Chapter V]{BCGGG} that the complex characteristic variety $\Xi_E\subset\bbP\bigl((E\otimes\bbC)^*\bigr)\simeq\mathbb{CP}^6$ is a subvariety of (complex) dimension $5$ for every integral element $E$.  Because of the $\Gtwo$-invariance of the system and the fact that $\Gtwo$ acts transitively on (real) $2$-planes in $\mathbb{R}^7$, the only possibility for the support of this characteristic variety for a given solution $(\sigma,\beta)$ on a neighborhood of a point $u\in\mathbb{R}^7$ is the set of complex co-vectors in $\bbP\bigl((T^*_u\mathbb{R}^7){\otimes}\bbC\bigr)$ that are null for $g_{\sigma_u}$.  In particular, the complex characteristic variety has no real points, implying that the associated PDE system is (overdetermined) elliptic. Consequently, 
$\lambda$-solitons that are $C^{1,\alpha}$ in harmonic coordinates 
(for some $\alpha>0$) must be real-analytic in those coordinates~\cite{Taylor1}.
\end{remark}

\section{The gradient case}

In this final section, I will discuss what is known about the special case that the $\lambda$-soliton~$(\sigma,X)$ is a \emph{gradient} soliton, i.e., $X = \nabla^{g_\sigma}f$ for some function $f$.   This case is interesting partly because, in the case of solitons for the Ricci flow, particularly in dimension $3$, the gradient solitons are the most important.

\begin{remark}[Ricci solitons in dimension $3$]
Just for comparison, it might be interesting to note that the generality analysis for general Ricci solitons in dimension~$3$, i.e., pairs $(g,X)$ of a metric $g$ and a vector field $X$ on a manifold $M^3$ that satisfy
\be
\mathrm{Ric}(g) = \lambda\,g + \mathcal{L}_Xg\,,
\ee
shows that, up to local diffeomorphism, the local solutions of this system depend on $s_2 = 6$ functions of two variables.  If one adds the `gradient condition' that $X = \nabla^gf$ for some function $f$, then one sees that, up to local diffeomorphism, the local solutions satisfying this extra condition depend on $s_2 = 2$ functions of two variables.  In both cases, the associated exterior differential system is involutive. See \cite[\S5.5]{Bryant2} for the gradient case.
\end{remark}

Many explicit examples of gradient $\lambda$-solitons are known, as can be seen by looking at the examples constructed in the articles~\cite{Ball2020,Fowdar,HaskinsNordstrom2021,Lauret2017,Lotay2017}.  However, as of this writing, it is not known what their generality is in the sense of \'E.~Cartan.  I will conclude this article by explaining where the difficulty lies.

\subsection{Formulation}
The gradient $\lambda$-soliton case, where $X = \nabla^{g_\sigma}f$ for some function~$f$, is locally equivalent to requiring that $X^\flat$, the $1$-form that is $g_\sigma$-dual to the vector field $X$, be closed.

Starting with the algebraic identity of $5$-forms
\be\label{5formidentity}
(X\lhk\sigma)\w\sigma = 2 X^\flat \w {\ast_\sigma}\sigma\,
\ee
and applying the exterior derivative 
to the left hand side of \eqref{5formidentity} gives
\be\label{5formidentity2}
\d\bigl((X\lhk\sigma)\w\sigma\bigr) = \d(X\lhk\sigma)\w\sigma 
= (\d\tau-\lambda\,\sigma)\w\sigma = 0,
\ee
since $\d({\ast_\sigma}\sigma) = \tau\w\sigma$ and $\d\sigma=0$ implies $\d\tau\w\sigma=0$.  Meanwhile applying the exterior derivative 
to the right hand side of of \eqref{5formidentity} (which is closed because
of \eqref{5formidentity2}) gives
\be
0 = 2\,\d(X^\flat)\w {\ast_\sigma}\sigma  - 2\,X^\flat\w \tau\w\sigma.
\ee
In particular, if $\d(X^\flat)=0$, 
which is the condition for a gradient $\lambda$-soliton, 
it follows that $X^\flat\w\tau\w\sigma = 0$, which is a algebraic relation on $X$ and $\tau$ that is equivalent to 
\be\label{Xtautorsion}
X\lhk \tau = 0,
\ee
i.e., the vanishing of a $1$-form that is bilinear in $X$ and $\tau$.

\subsection{Non-involutivity}
The condition $\d(X^\flat)=0$ 
is expressed in terms of $\beta$ as
\be\label{dXflat=0}
\d\bigl({\ast_\sigma}(\beta\w{{\ast_\sigma}\sigma})\bigr)=0.
\ee 

Thus, it is natural to study the gradient case by augmenting the differential ideal $\mathcal{I}_\lambda$ by adding to it the closed $2$-form $\d\zsf$, where $\zsf$ is the $1$-form on $X$ defined as
$$
\zsf_{(s,b)} = \pi^*\bigl({\ast_s}(b\w{{\ast_s}s})\bigr).
$$

Unfortunately, as \eqref{Xtautorsion} shows, this augmented differential ideal 
does not have any admissible integral elements at any~$(s,b)\in X$ 
for which the $6$-form $\bigl({\ast_s}(b\w{{\ast_s}s})\bigr)\w b \w s$ is nonzero.
Hence, the augmented ideal is not involutive.

The standard `next step' in such a case is to pull back the augmented differential ideal to the (nonsmooth) subvariety $Z\subset X$ on which $\zsf\w\bsf\w\ssf = 0$ and check that exterior differential system on $Z$ for involutivity.  Unfortunately, calculation shows that that exterior differential system is not involutive either, even on the smooth locus of~$Z$.

Two successive applications of the standard prolongation procedure to this system does not yield an involutive exterior differential system, and the calculations become increasingly complicated.  So far, no method has been found that leads to an involutive system whose solutions correspond to the gradient $\lambda$-solitons.

As a result, the generality of the gradient $\lambda$-solitons remains unknown 
as of this writing.  All we have is a collection of the various examples
already mentioned to show that gradient $\lambda$-solitons do, indeed, exist.

\bibliographystyle{hamsplain}

\providecommand{\bysame}{\leavevmode\hbox to3em{\hrulefill}\thinspace}

\end{document}